\documentclass[11pt,oneside,english,a4paper,hidelinks]{amsart}
\usepackage[T1]{fontenc}
\usepackage[utf8]{inputenc}
\usepackage[dvipsnames]{xcolor}
\usepackage[yyyymmdd,hhmmss]{datetime}
\usepackage[numbers,sort,compress]{natbib}
\usepackage{amstext,amsthm,amssymb}
\usepackage{xparse,mathrsfs,mathtools}
\usepackage[english]{babel}
\usepackage{csquotes}
\usepackage[unicode=true,%
	pdfusetitle,pdfdisplaydoctitle=true,%
	pdfpagemode=UseOutlines,%
	bookmarks=true,bookmarksnumbered=false,bookmarksopen=true,bookmarksopenlevel=2,%
	breaklinks=true,%
	pdfencoding=unicode,psdextra,
	pdfcreator={},
	pdfborder={0 0 0}
]{hyperref}
\usepackage{lmodern,microtype}
\usepackage[a4paper]{geometry}
\usepackage[capitalise,nameinlink]{cleveref}

\usepackage[a4paper]{geometry}

\usepackage{xifthen,ifthen}
\makeatletter
\newcounter{@ToDo}
\newcommand{\todo@helper}[1]{%
	({\color{blue}TODO~\arabic{@ToDo}: {#1\@addpunct{.}}})%
}
\newcommand{\todo}[1]{\stepcounter{@ToDo}%
	\relax\ifmmode\text{\todo@helper{#1}}%
	\else\todo@helper{#1}\fi%
}
\makeatother

\usepackage[dvipsnames]{xcolor}
\usepackage{tikz}
\usetikzlibrary{backgrounds,babel}
\newcommand{\ptSize}{2pt}

\newtheorem{thm}{Theorem}[section]
\newtheorem{prop}[thm]{Proposition}
\newtheorem{lem}[thm]{Lemma}
\newtheorem{cor}[thm]{Corollary}
\newtheorem*{thm*}{Theorem}
\newtheorem*{lem*}{Lemma}
\newtheorem{question}{Question}

\theoremstyle{remark}
\newtheorem{rem}[thm]{Remark}
\newtheorem*{rems*}{Remarks}

\DeclarePairedDelimiter\parentheses{\lparen}{\rparen}
\DeclarePairedDelimiter\floor{\lfloor}{\rfloor}
\DeclarePairedDelimiter\ceil{\lceil}{\rceil}
\DeclarePairedDelimiter\braces{\lbrace}{\rbrace}
\DeclarePairedDelimiter\abs{\lvert}{\rvert}
\DeclarePairedDelimiter\norm{\lVert}{\rVert}
\DeclarePairedDelimiter\ropeninterval{\lbrack}{\rparen}
\DeclarePairedDelimiter\lopeninterval{\lparen}{\rbrack}

\newcommand{\NN}{\mathbb{N}}
\newcommand{\ZZ}{\mathbb{Z}}
\newcommand{\RR}{\mathbb{R}}
\newcommand{\B}[1]{\mathcal{B}\parentheses{#1}}

\newcommand{\ModH}[1]{\mathscr{H}_{#1}}
\newcommand{\K}[6]{K_{#2,#3}\parentheses{#5,#6;#1,#4}}
\newcommand{\Kl}[3]{S\parentheses{#2,#3;#1}}
\DeclareMathOperator{\eOpname}{e}
\NewDocumentCommand\e{ s O{} m }{
	\IfBooleanTF{#1}{%
		\eOpname_{#2}\parentheses[\big]{#3}%
	}{\eOpname_{#2}\parentheses{#3}}%
}

\numberwithin{equation}{section}
\crefformat{equation}{(#2#1#3)}

\title{Modular hyperbolas and Beatty sequences}
\date{10th April~2019}


\author{Marc~Technau}
\address{Marc~Technau\\%
	Institute of Analysis and Number Theory\\%
	Graz University of Technology\\%
	Kopernikusgasse~24\\%
	8010~Graz\\%
	Austria}
\email{mtechnau@math.tugraz.at}

\begin{document}
\begin{abstract}
	Bounds for $\max\braces{m,\tilde{m}}$ subject to $m,\tilde{m} \in \ZZ\cap\ropeninterval{1,p}$, $p$ prime, $z$ indivisible by $p$, $m\tilde{m}\equiv z\bmod p$ and $m$ belonging to some fixed Beatty sequence $\braces{ \floor{ n\alpha+\beta } : n\in\NN }$ are obtained, assuming certain conditions on $\alpha$.
	The proof uses a method due to \citeauthor*{banks2006nonresidues}. As an intermediate step, bounds for the discrete periodic autocorrelation of the finite sequence $0,\, \e[p]{y\overline{1}}, \e[p]{y\overline{2}}, \ldots, \e[p]{y(\overline{p-1})}$ on average are obtained, where $\e[p]{t} = \exp(2\pi i t/p)$ and $m\overline{m} \equiv 1\bmod p$. The latter is accomplished by adapting a method due to \citeauthor{kloosterman1927ontherepresentation}.
\end{abstract}
\maketitle

\section{Introduction}
Consider the \emph{modular hyperbola}
\begin{equation}\label{eq:ModularHyperbola}
	\ModH{z\bmod p}
	= \braces{ (m,\tilde{m}) \in \ZZ^2\cap\ropeninterval{1,p}^2 : m\tilde{m}\equiv z\bmod p }
\quad (\text{with } p\nmid z),
\end{equation}
where the letter $p$ denotes a prime here and throughout.
The distribution of points on these modular hyperbolas has attracted wide interest and the interested reader is referred to~\cite{sharplinski2012modular-hyperbolas} for a survey of questions related to this topic and various applications.
Our particular starting point is the following intriguing property of such hyperbolas:
\begin{thm}
	\label{thm:LeastModHyperbolaPoint}
	For any prime $p$ and $z$ coprime to $p$ there is always a point $(m,\tilde{m})\in \ModH{z\bmod p}$ with $\max\braces{m,\tilde{m}}\leq 2\parentheses{\log p}p^{3/4}$.
\end{thm}
\begin{proof}
	See, e.g.,~\cite[p.~382]{heath-brown2000arithmeticappl}.\footnote{%
		\label{footnote:PointsInArbitraryBoxes}%
		The proof itself gives much more. Indeed, instead of restricting $(m,\tilde{m})$ to the square box $\lopeninterval{0,2(\log p) p^{3/4}}^2$, any rectangular box $\mathscr{R}\subseteq \lopeninterval{0,p}^2$ with sufficiently large area can be seen to have a non-empty intersection with $\ModH{z\bmod p}$.
	}
\end{proof}
Loosely speaking, the above theorem states that, when sampling enough points
\[
	(1,?),\, (2,?'),\, (3,?''),\, \ldots\, \in \ModH{z\bmod p},
\]
at least one amongst the second coordinates will not be too large.

Here we shall investigate analogous distribution phenomena when one of the coordinates is required to belong to some fixed Beatty set. The \emph{Beatty set} $\B{\alpha,\beta}$ associated with real \emph{slope} $\alpha\geq 1$ and \emph{shift} $\beta\geq 0$ is given by
\[
	\B{\alpha,\beta} = \braces{ \floor{ n\alpha+\beta } : n\in\NN },
\]
where $\floor{x}$ denotes the largest integer less than or equal to~$x$.
When thinking of $\B{\alpha,\beta}$ as ordered according to size of its elements, we call it a \emph{Beatty sequence}.
Such sequences appeared first in the astronomical studies of Johann~III Bernoulli~\cite{bernoulli1772sur-une} as a means to control the accumulative error when calculating successive multiples of an approximation to some given number.
Later, Beatty sequences were studied by Elwin Bruno Christoffel~\cite{christoffel1873observatio,christoffel1887lehrsaetze} with respect to their arithmetical significance with the goal of easing the discomfort present during that time when working with irrational numbers.
The name ,,Beatty sequence'' is in the honour of Samuel Beatty who popularised these sequences by posing a problem for solution in \emph{The American Mathematical Monthly}~\cite{beatty1926problem,beatty1927solutions}; the theorem to be proved there appears to be due to John William Strutt (3\textsuperscript{rd}~Baron Rayleigh)~\cite{rayleigh1926theory} though. 
A reader interested in more recent work on Beatty sequences is invited to start his or her journey by tracing the references in~\cite{abercrombie2009arithm}, for instance.

\begin{figure}[!ht]\centering%
	\begingroup%
	\begin{tikzpicture}[scale=.14]
		\foreach \x/\y in {1/46, 2/23, 3/31, 4/35, 5/28, 6/39, 7/20, 8/41, 9/26, 10/14, 11/17, 12/43, 13/18, 14/10, 15/25, 16/44, 17/11, 18/13, 19/42, 20/7, 21/38, 22/32, 23/2, 24/45, 25/15, 26/9, 27/40, 28/5, 29/34, 30/36, 31/3, 32/22, 33/37, 34/29, 35/4, 36/30, 37/33, 38/21, 39/6, 40/27, 41/8, 42/19, 43/12, 44/16, 45/24, 46/1}{
			\coordinate (PT-\x) at (\x,\y);
			\global\edef\maxX{\x}
			\pgfmathsetmacro{\primeFromFigureTMP}{int(\x+1)}
			\global\edef\primeFromFigure{\primeFromFigureTMP}
		}
	
		\pgfmathsetmacro{\ALPHA}{pi}
		\pgfmathsetmacro{\BETA}{e}
		\pgfmathsetmacro{\maxN}{int(\primeFromFigure/\ALPHA)}
		\begin{pgfonlayer}{background}
			\path (-3,1) -- ({3+\maxX},1);
			\draw (0,0) rectangle ++(\maxX,\maxX);
			\draw[very thick] (0,0) rectangle ++(18,18);
		\end{pgfonlayer}
		\foreach \n in {1,2,3,...,\maxN}{
			\pgfmathsetmacro{\y}{int(floor(\n*\ALPHA+\BETA))}
			\coordinate (TMP) at (PT-\y);
			\coordinate (PT-\y) at (PT-1);
			\begin{pgfonlayer}{background}
				\fill[black,shift only] (TMP) circle (\ptSize);
			\end{pgfonlayer}
		}
		\begin{pgfonlayer}{background}
			\draw[shift only,fill=white] \foreach \x/\y in {1,2,3,...,\maxX}{ (PT-\x) circle (\ptSize) };
		\end{pgfonlayer}
		\draw[thick,line cap=round]
			(0,0) -- ++(0,-1) node[below] {\footnotesize$0$}
			(18,0) -- ++(0,-1) node[below] {\footnotesize$18$}
			(\maxX,0) -- ++(0,-1) node[below] {\footnotesize$\maxX$}
		;
	\end{tikzpicture}%
	\endgroup%
	\caption{An illustration of the points $(m,\tilde{m})\in\mathscr{H}_{-1\bmod\primeFromFigure}$. The points $(m,\tilde{m})$ with $m$ belonging to the Beatty set $\B{\pi,e}$ are filled.}
	\label{fig:ModHyperbolaSmallSolutions}%
\end{figure}
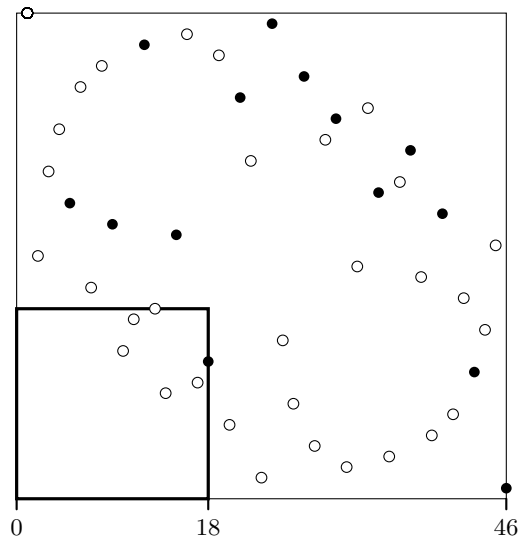

Returning to modular hyperbolas, the question we seek to answer may be enunciated as follows:
\begin{question}
	\label{question:LeastModHyperbolaPoint:Beatty}
	If the first coordinate of $(m,\tilde{m})$ in \cref{thm:LeastModHyperbolaPoint} is additionally required to belong to a fixed Beatty set with irrational slope, can one still prove a result like \cref{thm:LeastModHyperbolaPoint}, i.e., is it impossible for a Beatty set to contain only those elements $m<p$ with \enquote{large} corresponding $\tilde{m}$ from~\cref{eq:ModularHyperbola}?
\end{question}
More specifically, for fixed irrational $\alpha>1$ and non-negative $\beta$, we shall be interested in solving
\begin{equation}\label{eq:SolveMe}
	m\tilde{m}\equiv z\mod p
	\quad\text{with}\quad
	m\in\B{\alpha,\beta}
	\quad\text{and}\quad
	1\leq m,\tilde{m}<p
\end{equation}
whilst keeping $\max\braces{m,\tilde{m}}$ as small as possible. Thus, we shall want to bound
\begin{equation}\label{eq:SolveMe:F(zp)}
	F(z\bmod p) = \min\braces{ \max\braces{m,\tilde{m}} : (m,\tilde{m}) \text{ such that~\cref{eq:SolveMe} holds} }.
\end{equation}

As an illustration, consider \cref{fig:ModHyperbolaSmallSolutions}: therein, bounding $F(-1\bmod 47)$ is equivalent to asking how large one must take the side length of a square with lower left corner positioned at $(0,0)$, in order to be guaranteed to find a black point inside. The smallest such square is sketched thick in the figure.

The main result of this paper is that, for any fixed $\alpha$ satisfying certain Diophantine conditions and arbitrary (not necessarily fixed) non-negative $\beta$, the quantity $F(z\bmod p)$ is `not too large' in the sense that
\begin{equation}\label{eq:smallestSol:FirstVersion}
	F(z\bmod p) \leq \alpha p^{684/727}\log p+\beta
\end{equation}
for all sufficiently large $p$ (in terms of $\alpha$) and arbitrary $z$ indivisible by $p$. (Here the main point is that the exponent of $p$ in~\cref{eq:smallestSol:FirstVersion} is strictly smaller than~$1$.)
We refer to \cref{thm:max:mm:bound} for the precise formulation.

\section{Main results}
\label{sec:MainResults}

\subsection{A partial answer to \cref{question:LeastModHyperbolaPoint:Beatty}}
In order to describe the numbers $\alpha$ for which we obtain results, recall that the \emph{type} $\tau$ of an irrational number $\alpha$ is defined by
\[
	\tau = \sup\braces[\Big]{ \eta\in\RR : \liminf_{q\to\infty} q^{\eta} \norm{ \alpha q } =0 },
\]
where $\norm{\varrho} = \min_{n\in\ZZ} \abs{\varrho - n}$ denotes the distance to a nearest integer.
(Note that Dirichlet's approximation theorem ensures that $\tau\geq 1$.)
We say that $\alpha$ is of \emph{finite type} if $\tau < \infty$.

Our main result may now be enunciated as follows:
\begin{thm}
	\label{thm:max:mm:bound}
	Let $\epsilon>0$ and $\beta$ be non-negative. Moreover, assume that $\alpha>1$ is irrational and of finite~$\text{type}\leq\frac{512}{43}-\epsilon$. Then there is a number $p_0(\alpha,\epsilon)$ such that, for all primes $p\geq p_0(\alpha,\epsilon)$ and $z$ coprime to $p$, there is a point $(m,\tilde{m})\in \ModH{z\bmod p}$ with $m\in\B{\alpha,\beta}$ and
	\begin{equation}\label{eq:smallestSol}
		\max\braces{m,\tilde{m}} \leq \alpha p^{684/727}\log p+\beta.
	\end{equation}
\end{thm}

\begin{rem}\label{rem:Improvements}
	After posting an initial draft of this article on \enquote{the arXiv}, I.~E.~Shparlinski kindly pointed out to the author that a generalised form of Weil's bound also applies directly to the sum~\cref{eq:Sum:Sxywp}. This leads to a chain of improvements in~\cref{prop:KloostermanTypeEstimate}, \cref{thm:KloostermanBeattyBound}, and \cref{thm:max:mm:bound}.
	We sketch the relevant changes in~\cref{sec:Improvements} and close with a discussion of two directions in which our results may potentially be generalised.
\end{rem}

\begin{rem}
	If $(m,\tilde{m})$ is a solution of~\cref{eq:SolveMe}, then
	\(
	m \geq \min\B{\alpha,\beta} = \floor{\alpha+\beta} > \beta
	\).
	Hence, the dependence of the right hand side of~\cref{eq:smallestSol} on $\beta$ is not a deficiency in the proof, but instead occurs naturally.
\end{rem}

\subsection{Intermediate results and discussion of the methods}
The proof of \cref{thm:max:mm:bound} is carried out in \cref{sec:Proofs} via a discrete Fourier transform method.
It is an adaptation of the argument which is used in~\cite{heath-brown2000arithmeticappl} to prove \cref{thm:LeastModHyperbolaPoint}.
Here one naturally encounters the discrete Fourier transform of the characteristic function of points whose (integer) coordinates are inverse to each other modulo~$p$, namely \emph{Kloosterman sums}:
\begin{equation}\label{eq:KloostermanSum}
	\Kl{p}{x}{y} = \sum_{1\leq m<p}\e[p]{xm+y\overline{m}}.
\end{equation}
(Here $x$ and $y$ are integers, $\e[p]{r} = \exp(2\pi i\, r/p)$, and $\overline{m}$ is the unique positive integer below $p$ which is inverse to $m$ modulo~$p$.)
These sums were first introduced by~\citet{kloosterman1927ontherepresentation} in his seminal refinement of the Hardy--Littlewood circle method to handle diagonal quadratic forms in four variables.
He proved the bound
\[
	\sum_{1\leq m<p} \e[p]{xm+y\overline{m}}
	\ll p^{3/4} \gcd(x,y,p)^{1/4},
\]
which was later improved by \citet{weil1948onsomeexp} to
\[
	\abs[\bigg]{ \sum_{1\leq m<p}\e[p]{xm+y\overline{m}} }
	\leq 2p^{1/2} \gcd(x,y,p)^{1/2},
\]
the latter bound being asymptotically optimal (see~\cite[Section~11.7]{iwaniec2004analytic}).

Upon using such a Fourier transform approach to bound~\cref{eq:SolveMe:F(zp)} without imposing the restriction $m\in\B{\alpha,\beta}$ in~\cref{eq:SolveMe}, one naturally has to deal with \emph{incomplete} Kloosterman sums:
\[
	\sum_{1\leq m<M}\e[p]{xm+y\overline{m}}
	\quad(\text{with } 1\leq M\leq p).
\]
Using a completing technique, bounds for such sums can be derived from bounds for the complete sum~\cref{eq:KloostermanSum}.
In our setting---reinstating the restriction $m\in\B{\alpha,\beta}$ in~\cref{eq:SolveMe}---one has to deal with incomplete Kloosterman sums along the Beatty sequence in question:
\begin{equation}\label{eq:KloostermanBeatty}
	\K{p}{\alpha}{\beta}{N}{x}{y}
	= \sum_{\substack{ 1\leq n\leq N \\ p\nmid \floor{n\alpha+\beta} }} \e*[p]{x\floor{n\alpha+\beta} + y\overline{\floor{n\alpha+\beta}}}.
\end{equation}
We are able to bound such sums (see \cref{thm:KloostermanBeattyBound} below) by means of a method due to \citet*{banks2006nonresidues,banks2006short-character-sums,banks2007prime-divisors}.
In order to succeed in our particular case, the method requires bounds for sums of the shape
\begin{equation}\label{eq:Sum:Sxywp}
	S(x,y,w;p) = \sum_{\substack{ 1\leq m<p \\ p \nmid (m+w) }} \e[p]{ x\overline{m}+y\overline{m+w} }
\end{equation}
and this input is obtained by adapting Kloosterman's original method~\cite{kloosterman1927ontherepresentation} for bounding his sums (see \cref{prop:KloostermanTypeEstimate} below, and \cref{rem:Kloosterman} for more context on the similarity to Kloosterman's argument).
The quality of the bounds we obtain depends on the Diophantine properties of the slope $\alpha$ of the Beatty sequence $\B{\alpha,\beta}$.

\subsection{Plan of the paper}
\label{sec:PlanOfPaper}
The rest of the paper is structured as follows:
first, we fix some notation and recall some basic facts related to Diophantine properties of irrational numbers $\alpha$.
Then, in \cref{sec:KloostermanBeattyBounds}, we state our results concerning Kloosterman sums along Beatty sequences~\cref{eq:KloostermanBeatty}.
In \cref{sec:Proofs}, we provide the proofs for our results.
Finally, in \cref{sec:Improvements}, we sketch how the improvements hinted at in \cref{rem:Improvements} can be obtained by adjusting the relevant parts of our arguments.
The last part of \cref{sec:Improvements} contains a discussion of potential generalisations of our work.

\section{Preliminaries}

\subsection{Notation}

Before proceeding, we fix some notation used throughout the paper.
We use the Vinogradov symbol $f(x) \ll g(x)$ and the Big Oh notation $f(x) = O(g(x))$ to mean $\abs{f(x)} \leq C g(x)$ for some absolute constant $C>1$. In cases where the constant depends on a parameter (usually $\alpha$ and sometimes $\epsilon>0$), we indicate this with subscripts.

\subsection{Discrepancy}
\label{sec:DiscrepancyFacts}
For a real number $\varrho$ we denote its \emph{fractional part} by $\braces{\varrho} = \varrho - \floor{\varrho} \in \lbrack0,1\rparen$.
We denote the \emph{discrepancy} of the finite sequence $(\braces{n\alpha+\beta})_{n\leq N}$ by
\begin{equation}\label{eq:Discrepancy}
	D_{\alpha,\beta}(N)
	= \sup_{0\leq x<y<1} \abs*{ \frac{\#\braces{ n\leq N : \braces{n\alpha+\beta} \in \ropeninterval{x,y} }}{N} - (y-x) },
\end{equation}
and we put $D_{\alpha}(N) = D_{\alpha,0}(N)$.
The discrepancy~\cref{eq:Discrepancy} is a measure for how well distributed the finite sequence $(\braces{n\alpha+\beta})_{n\leq N}$ is.
A drawback of the above definition is that it is needlessly dependent on the shift $\beta$ (consider, for instance, $(\alpha,N)=(\frac{15}{8},2)$ and $\beta\in\braces{0,\frac{1}{8}}$).
This dependency could be removed by defining the discrepancy as, e.g., \citet{montgomery1994} does, but we refrain from doing so for the purpose of being able to use results from~\cite{kuipers1974uniformdistribution} below.
Nonetheless, the dependence of~\cref{eq:Discrepancy} on $\beta$ can be easily sidestepped by means of the following well-known fact:
\begin{lem}\label{lem:DiscrepBetaIndep}
	For any $\alpha,\beta\in\RR$ and $N\in\NN$, we have
	\(
		D_{\alpha,\beta}(N) \leq 8 D_{\alpha}(N)
	\).
\end{lem}
\begin{proof}
	This follows easily from {\citep[Chapter~2, Theorem~1.3]{kuipers1974uniformdistribution}} and {\citep[Chapter~1, Eq.~(8)]{montgomery1994}}.
\end{proof}
In view of the above, it suffices to consider $D_{\alpha}(N)$ in the sequel.
It is a classic result due to Bohl, {Sierpi\'nski} and Weyl that the discrepancy of $(\braces{n\alpha})_{n\leq N}$ tends to zero (see~\cite{weyl1916gleichverteilung} and the references therein). Namely, we have:
\begin{lem}[{\cite[Satz~2]{weyl1916gleichverteilung}}]
	\label{lem:Weyl}
	Let $\alpha$ be irrational. Then
	\[
		D_{\alpha}(N) = o_\alpha(1),
		\quad\text{as $N\to\infty$},
	\]
	where the implied constant only depends on $\alpha$.
\end{lem}

For $\alpha$ of finite type as defined in \cref{sec:MainResults}, \cref{lem:Weyl} can be sharpened considerably:
\begin{lem}[{\citep[Chapter~2, Theorem~3.2]{kuipers1974uniformdistribution}}]
	\label{lem:FiniteTypeDiscrepancyBound}
	Let $\alpha$ be of finite type $\tau$. Then, for every $\epsilon>0$,
	\[
		D_{\alpha}(N)
		\ll_{\alpha,\epsilon} N^{-1/\tau+\epsilon},
	\]
	where the implied constant only depends on $\alpha$ and $\epsilon$.
\end{lem}

\section{Bounds for Kloosterman sums along Beatty sequences}
\label{sec:KloostermanBeattyBounds}

In this section we state our bounds for~\cref{eq:KloostermanBeatty} (see \cref{thm:KloostermanBeattyBound} below). However, before doing so, we give a short survey of previous work on sums along Beatty sequences.
In fact, as a general heuristic, for any fixed Beatty sequence $\B{\alpha,\beta}$ and \enquote{reasonable} arithmetical function $f$, one should expect that
\begin{equation}\label{eq:fBeattyAverage}
	\sum_{\substack{1\leq m \leq x \\ \mathclap{m\in\B{\alpha,\beta}} }} f(m)
	= \frac{1}{\alpha}\sum_{1\leq m \leq x} f(m) + \braces{\text{error term}}
	\quad\text{as $x\to\infty$},
\end{equation}
with some error term which one expects to be able to bound non-trivially.
Indeed, this heuristic principle is substantiated by a sizable list of particular results:
\begin{itemize}
	\item $d$, the \emph{divisor function} (i.e., the function giving the number of positive divisors of its argument), starting with \citet{abercrombie1995beatty-sequences}, improved by \citet{begunts2004on-an-analogue}, and work on generalised divisor functions by \citet{zhai1997a-note-on-a} and \citet{lu2004-the-divisor-problem}.
	\item Certain multiplicative functions like $n\mapsto n^{-1}\varphi(n)$, where $\varphi$ is \emph{Euler's totient}, due to \citet{begunts2005on-the-distribution}, and a certain class of multiplicative $f$ in the work of \citet{guloglu2008sums-of-multiplicative}.
	\item Dirichlet-characters and special functions related to the orbits of elements $g\bmod m$ along Beatty sequences, treated by \citet{banks2006nonresidues} (see also~\cite{banks2006short-character-sums}), with improvements when $f$ is the Legendre-Symbol due to \citet*{banks2008density-of-non-residues}.
	\item $\omega$, the function counting the number of distinct prime divisors of its argument (without multiplicity), and $n\mapsto (-1)^{\Omega(n)}$, where $\Omega(n)$ counts the number of distinct prime divisors of $n$ with multiplicity, due to \citet{banks2007prime-divisors}.
	\item $\Lambda$, the \emph{von Mangoldt lambda function}, studied by \citet{banks2009primenumbers} who attribute earlier results to \citet[Chapter~4.V]{ribenboim1996new-book}, although it seems that such observations were already apparent to Heilbronn in 1954, as evidenced in~\cite[Notes to Chapter~XI]{vinogradov2004themethod}.
\end{itemize}
Moreover, by a result due to \citet*{abercrombie2009arithm}, the heuristic principle is seen to hold in surprising generality for a large class of functions $f$ in a metric sense.

However, the problem of bounding~\cref{eq:KloostermanBeatty} is a little more subtle, as the function
\[
	f(m) =
	\begin{cases}
		\e[p]{xm+y\overline{m}} & p\nmid m, \\
		0 & p\mid m,
	\end{cases}
\]
for which we might want to invoke some result of the type given in \cref{eq:fBeattyAverage}, also depends on $x$, $y$ and $p$, and we lack the necessary uniformity in those parameters.
Nonetheless, we obtain the following:
\begin{thm}
	\label{thm:KloostermanBeattyBound}
	Suppose that $p$ is a prime, $\beta\geq0$ and $x,y\in\ZZ$ such that $p\nmid y$. Then, for any irrational $\alpha>1$ and $N\leq p$, the sum $\K{p}{\alpha}{\beta}{N}{x}{y}$ given by~\cref{eq:KloostermanBeatty} satisfies the bound
	\begin{equation}\label{eq:KloostermanBeattyBound}
		\abs{\K{p}{\alpha}{\beta}{N}{x}{y}}
		\ll_\alpha N^{297/512} p^{43/128} + ND_{\alpha}(N),
	\end{equation}
	where the implied constant only depends on $\alpha$.
\end{thm}
The proof of this result is given in \cref{sec:Proofs} and makes use of an estimate for the \emph{periodic autocorrelation} of the finite sequence
\[
	0,\, \e[p]{y\overline{1}},\, \e[p]{y\overline{2}},\, \ldots,\, \e[p]{y(\overline{p-1})}
\]
\emph{on average}.
Such a bound is furnished by the next result applied with $x=-y$:
\begin{prop}
	\label{prop:KloostermanTypeEstimate}
	Suppose that $p$ is a prime and $x,y\in\ZZ$ such that $p\nmid y$. If $S(x,y,w;p)$ is given by~\cref{eq:Sum:Sxywp}, then, for any set $\mathscr{W}$ of pairwise incongruent numbers modulo~$p$, we have
	\[
	\sum_{w\in\mathscr{W}}\abs{S(x,y,w;p)}
	\ll (\#\mathscr{W})^{3/4}p.
	\]
\end{prop}


Turning back to \cref{thm:KloostermanBeattyBound} and recalling its assumption that $N\leq p$, we find that the first term on the right hand side of~\cref{eq:KloostermanBeattyBound} is dominant, provided that the type of $\alpha$ is strictly smaller than $\frac{512}{43}$. As an immediate consequence we may state:
\begin{cor}
	\label{cor:KloostermanBeattyBound:FiniteType}
	Let $\epsilon>0$. Under the hypotheses of \cref{thm:KloostermanBeattyBound}, and restricting to only those irrational $\alpha>1$ of finite~$\text{type}<\frac{512}{43}$, we have the bound
	\begin{equation}
		\label{eq:KloostermanBeattyBound:FiniteType}
		\abs{\K{p}{\alpha}{\beta}{N}{x}{y}}
		\ll_{\alpha} N^{297/512} p^{43/128}.
	\end{equation}
\end{cor}
In order to put this into context, note that the famous result of \citet{khintchine1924einigesaetze} asserts that almost all real numbers are of type~$1$ and the celebrated theorem of \citet{roth1955rational} establishes that all real algebraic irrational numbers are of type~$1$; for numbers of finite type larger than~$1$, the Lebesgue measure fails to provide useful information about their abundance.
However, by the Jarn\'ik--Besicovitch theorem~\citep{jarnik1929diophantische,besicovitch1934setsIV}, we know that the Hausdorff dimension of all real numbers of type $>\tau$ is
\[
\dim_{\mathrm{H}}\braces{x\in\RR\text{ of type}>\tau}=\frac{2}{\tau+1}.
\]
In particular, bounds for the Hausdorff dimension of the set of those $\alpha$, for which we obtain no results, are immediate.\\
Furthermore, observe that \cref{thm:KloostermanBeattyBound} is only non-trivial for $p$ and $N$ in some range of the type
\[
	N\leq p\ll_\epsilon N^{5/4-\epsilon}.
\]
However, in view of \cref{lem:Weyl}, for fixed irrational $\alpha$, \cref{thm:KloostermanBeattyBound} always beats the trivial bound $\abs{\K{p}{\alpha}{\beta}{N}{x}{y}} \leq N$ in the above range provided that $N$ is large enough.

\section{Proofs}
\label{sec:Proofs}

\subsection{Proof of \texorpdfstring{\cref{thm:max:mm:bound}}{Theorem\autoref{thm:max:mm:bound}}}
We start by fixing some irrational $\alpha>1$ and non-negative $\beta$.
To detect pairs $(m,\tilde{m})$ satisfying~\cref{eq:SolveMe}, we adapt the arguments given in~\cite[p.~382]{heath-brown2000arithmeticappl}. Taking
\begin{equation}\label{eq:IntervalCharacteristicFunction}
	f_M(\tilde{m}) = \begin{cases}
		1 & \text{if there exists $1\leq m\leq M$ such that~\cref{eq:SolveMe} holds}, \\
		0 & \text{otherwise},
	\end{cases}
\end{equation}
one immediately observes that $F(z\bmod p)$ from~\cref{eq:SolveMe:F(zp)} may be written as
\[
	F(z\bmod p) = \min\braces[\bigg]{ 1\leq M<p : \sum_{1\leq \tilde{m}\leq M} f_M(\tilde{m}) > 0 }.
\
\]
Therefore, in order to bound $F(z\bmod p)$, it suffices to establish the positivity of
\begin{equation}\label{eq:BoxCounter}
	\sum_{1\leq \tilde{m}\leq M} f_M(\tilde{m})
\end{equation}
for some $M\in\ropeninterval{\alpha+\beta,p}$, which one would like to be as small as possible.
To achieve this, one notes that~\cref{eq:BoxCounter} should be well approximated by the average
\begin{equation}\label{eq:BoxCounter:Extended}
	\frac{M}{p} \sum_{1\leq \tilde{m}<p} f_M(\tilde{m})
	= \frac{M}{p} \#\braces{ m\leq M : m\in\B{\alpha,\beta} }
	> \frac{M(M-\alpha-\beta)}{p\alpha}.
\end{equation}
This can be made precise by a standard completing technique. We employ this in the following form given in~\cite{heath-brown2000arithmeticappl}:
\begin{lem}\label{lem:CompletingTechnique}
	Let $(A_m)_m$ be a sequence of complex numbers with period $q$, and let $\hat{A}_k$ be the discrete Fourier transform $\sum_{m=1}^q A_m \e[q]{mk}$. Then, for any integers $a<b$, we have
	\[
		\abs[\bigg]{ \sum_{a<m\leq b} A_m - \frac{b-a}{q} \hat{A}_0 }
		\leq (\log q) \max_{1\leq k < q} \abs{ \hat{A}_k }.
	\]
\end{lem}
\begin{proof}
	A slightly more general statement along with its proof can be found in~\cite[Lemma~12.1]{iwaniec2004analytic}.
\end{proof}
Returning to our analysis of~\cref{eq:BoxCounter} and applying \cref{lem:CompletingTechnique}, we infer that
\begin{equation}\label{eq:DerivationFromExpectation}
	\abs[\bigg]{
		\sum_{1\leq \tilde{m}\leq M} f_M(\tilde{m})
		- \frac{M}{p} \sum_{1\leq \tilde{m}<p} f_M(\tilde{m})
	}
	\leq \parentheses{\log p} \max_{1\leq k<p} \abs[\bigg]{
		\sum_{1\leq \tilde{m}<p} f_M(\tilde{m}) \e[p]{k\tilde{m}}
	}.
\end{equation}
Clearly, if the right hand side of~\cref{eq:DerivationFromExpectation} does not exceed the right hand side of~\cref{eq:BoxCounter:Extended}, then~\cref{eq:BoxCounter} is guaranteed to be positive.
For $1\leq k<p$, we have
\[
	\sum_{1\leq \tilde{m}<p} f_M(\tilde{m}) \e[p]{k\tilde{m}}
	= \sum_{1\leq \tilde{m}<p} \sum_{\substack{ 1\leq m\leq M \\ m\in\B{\alpha,\beta} \\ \mathclap{m\tilde{m}\equiv z\bmod p} }} \e[p]{k\tilde{m}}
	= \sum_{\substack{ 1\leq m\leq M \\ m\in\B{\alpha,\beta} }} \e[p]{kz\overline{m}}.
\]
Thus, in conclusion, if $M<p$ is a positive integer such that
\begin{equation}\label{eq:HyperbolaSolutionDetector}
	\max_{1\leq y<p} \abs[\bigg]{
		\sum_{\substack{ 1\leq m\leq M \\ m\in\B{\alpha,\beta} }} \e[p]{y\overline{m}}
	}
	\leq \frac{M\parentheses{M-\beta-\alpha}}{p (\log p) \alpha},
\end{equation}
then $F(z\bmod p)\leq M$. (It is perhaps worthwhile to observe at this point that the bound on $F(z\bmod p)$ produced by this method is wholly independent of~$z\not\equiv 0\bmod p$. Indeed, this is a recurring feature in many problems which are similar to \cref{question:LeastModHyperbolaPoint:Beatty}, as is perhaps already foreshadowed by the footnote on page~\pageref{footnote:PointsInArbitraryBoxes}. We refer to~\cite{sharplinski2012modular-hyperbolas} for more information.)

We now put $M=\floor{N\alpha+\beta}$ (with $N\geq 3$ as to have $M\geq\alpha+\beta$, as required earlier) in~\cref{eq:HyperbolaSolutionDetector}  and employ \cref{cor:KloostermanBeattyBound:FiniteType}, tacitly assuming its hypotheses in the process.
Then, after a short calculation, we find that~\cref{eq:HyperbolaSolutionDetector} is satisfied provided that
\begin{equation}\label{eq:HyperbolaSolutionDetector:Simplified}
(\log p) N^{-2+297/512} p^{1+43/128}
\ll_{\alpha,\epsilon} 1.
\end{equation}
Letting
\[
N = \floor{ p^{684/727}\log p },
\]
the condition~\cref{eq:HyperbolaSolutionDetector:Simplified} will be satisfied for $p$ sufficiently large in terms of $\alpha$ and $\epsilon$.
This completes the proof of \cref{thm:max:mm:bound}.

\subsection{Proof of \texorpdfstring{\cref{prop:KloostermanTypeEstimate}}{Proposition\autoref{prop:KloostermanTypeEstimate}}}

We start by deducing \cref{prop:KloostermanTypeEstimate} from the following result:
\begin{prop}
	\label{prop:KloostermanTypePowerEstimate}Under the assumptions of \cref{prop:KloostermanTypeEstimate}, we have
	\begin{equation}
		\sum_{w\in\mathscr{W}}\abs{S(x,y,w;p)}^{4}\ll p^{4}.\label{eq:KloostermanTypePowerEstimate}
	\end{equation}
\end{prop}
\begin{proof}[Proof of \cref{prop:KloostermanTypeEstimate}.]
	Note that, by Hölder's inequality, we have
	\[
		\sum_{w\in\mathscr{W}}\abs{S(x,y,w;p)} \leq (\#\mathscr{W})^{3/4}\cdot\biggl(\sum_{w\in\mathscr{W}}\abs{S(x,y,w;p)}^{4}\biggr)^{1/4}.
	\]
	The result now follows after applying \cref{prop:KloostermanTypePowerEstimate}.
\end{proof}

To prove \cref{prop:KloostermanTypePowerEstimate}, we adapt Kloosterman's original argument for bounding his sums (see~\citep[Section~2.43]{kloosterman1927ontherepresentation} and the comments made in \cref{rem:Kloosterman} below).
The argument is based on using a transformation property of the sums $S(x,y,w;p)$ (see~\cref{eq:TransformationProperty} below) in order to find a large contribution of $\abs{S(x,y,w;p)}^{4}$ in an average over the first three parameters (see~\cref{eq:ContributionToAverage}).
This average can be seen to count the number of solutions to certain congruence equations (see~\cref{eq:CountingSolutions}), and this number can be bounded non-trivially.

We start with an analysis of the congruence equations in question. To this end, write $\boldsymbol{m}$ for an integer vector $(m_1,m_2,m_3,m_4)$. For any integer $u$, we let
\begin{equation}\label{eq:Adef}
	A(\boldsymbol{m},u)
	= \sum_{1\leq k\leq 4} (-1)^{k} \overline{m_k+u}
\end{equation}
if $p\nmid(m_1+u)(m_2+u)(m_3+u)(m_4+u)$ and, for notational simplicity, we put $A(\boldsymbol{m},u)=1$ otherwise. The key result is the following:
\begin{lem}
	\label{lem:Diophantine}
	Let $p$ be a prime and $A(\boldsymbol{m},u)$ be given by~\cref{eq:Adef} and suppose that
	\begin{equation}\label{eq:Xu:def}
		\mathscr{X}_{u}
		= \braces{ \boldsymbol{m}\in\braces{1,\ldots,p}^{4} : A(\boldsymbol{m},0)\equiv A(\boldsymbol{m},u)\equiv0\bmod p }.
	\end{equation}
	Then
	\[
		\sum_{1\leq u\leq p}\#\mathscr{X}_{u}\ll p^{3}.
	\]
\end{lem}
\begin{proof}
	Consider $\boldsymbol{m} = (m_1,m_2,m_3,m_4)\in\mathscr{X}_{u}$. Certainly we have
	\begin{equation}\label{eq:spDivisibility}
		p\nmid\prod_{1\leq k\leq4}m_k(m_k+u).
	\end{equation}
	Additionally, from $p\mid A(\boldsymbol{m},0)$ and $p\mid A(\boldsymbol{m},u)$,
	we deduce that
	\begin{equation}\label{eq:sEqs}
		\left\lbrace
		\begin{alignedat}{2}
			\overline{m_1}+\overline{m_3} &
			\equiv \overline{m_2}+\overline{m_4} &
			\mod p, \\
			\overline{m_1+u}+\overline{m_3+u} &
			\equiv \overline{m_2+u}+\overline{m_4+u} &
			\mod p.
		\end{alignedat}
		\right.
	\end{equation}
	This is satisfied trivially if $(m_1,m_3)$ is a permutation of $(m_2,m_4)$. There are at most $2(p-1)^{2}$ such trivial solutions. Assume next that $\boldsymbol{m}$ is a non-trivial solution to the system~\cref{eq:sEqs} and, additionally, that neither expression in~\cref{eq:sEqs} is $\equiv0$; this additional assumption excludes at most $2(p-1)^{2}$ values of $\boldsymbol{m}$.
	From~\cref{eq:sEqs} it follows that
	\begin{align*}
		\MoveEqLeft
		(\overline{m_1}+\overline{m_3}) m_1m_3 (m_2+m_4) \\ &
		\equiv (m_2+m_4)(m_1+m_3) \\ &
		\equiv (m_1+m_3) m_2m_4 (\overline{m_2}+\overline{m_4})
		\mod p.
	\end{align*}
	Upon canceling $\overline{m_1} + \overline{m_3} \equiv \overline{m_2} + \overline{m_4} \not \equiv 0 \bmod p$, we find that
	\[
		m_1m_3 (m_2+m_4)
		\equiv (m_1+m_3) m_2m_4 \mod p,
	\]
	which in turn may be rearranged to
	\begin{gather}\label{eq:CongEq:1}
		m_1(m_2m_3+m_3m_4-m_2m_4)
		\equiv m_2m_3m_4 \mod p.
	\end{gather}
	Note that the term $m_2m_3+m_3m_4-m_2m_4\bmod p$ cannot vanish, for otherwise~\cref{eq:CongEq:1} implies $p\mid m_2m_3m_4$, in contradiction to~\cref{eq:spDivisibility}. Hence, for any $\boldsymbol{m}$ satisfying the above assumptions, we may compute $m_1\bmod p$ from $(m_2,m_3,m_4)$. Indeed, by~\cref{eq:CongEq:1},
	\[
		m_1 \equiv m_2m_3m_4 \overline{ (m_2m_3 + m_3m_4 - m_2m_4) } \mod p.
	\]
	Next, we claim that $\boldsymbol{m}$ belongs to at most two of the sets $\mathscr{X}_{u}$ ($u=1,\ldots,p$). To see this, first observe that, along similar lines as the deduction of~\cref{eq:CongEq:1} from~\cref{eq:sEqs}, we have
	\[
		\begin{aligned}
			\MoveEqLeft
			(m_1+u)(m_3+u)(m_2+m_4+2u) \\ &
			\equiv(m_1+m_3+2u)(m_2+u)(m_4+u) \mod p.
		\end{aligned}
	\]
	Recalling that $\boldsymbol{m}$ was assumed to be non-trivial, the difference of the left and right hand side of the above congruence is seen to be a \emph{non-zero} polynomial of degree at most two in $u$. Since $\ZZ/p\ZZ$ is an integral domain, the claim follows.
	From this, and taking the trivial solutions into account, we conclude that
	\[
		\sum_{u<p} \#\mathscr{X}_{u}
		\leq p (2p^{2} + 4(p-1)^2)
		\ll p^3.
		\qedhere
	\]
\end{proof}
\begin{proof}[Proof of \cref{prop:KloostermanTypePowerEstimate}.]
	First recall~\cref{eq:Sum:Sxywp} and observe that 
	\begin{equation}\label{eq:TransformationProperty}
		S(x,y,w;p) = S(ax,ay,aw;p)
	\end{equation}
	for any $a$ indivisible by $p$. Thus, the average
	\[
		\varSigma = \mathop{\sum\sum\sum}_{1\leq r,t\leq p;\, 1\leq u<p} \abs{S(r,t,u;p)}^4
	\]
	contains $p-1$ copies of $\abs{S(x,y,w;p)}^{4}$. Hence,
	\begin{equation}\label{eq:ContributionToAverage}
		(p-1) \sum_{w\in\mathscr{W}} \abs{S(x,y,w;p)}^4
		\leq \varSigma.
	\end{equation}
	To estimate $\varSigma$, write $\boldsymbol{m}=(m_1,m_2,m_3,m_4)$ and start by squaring out each term $\abs{S(r,t,u;p)}^{4}$ twice to obtain
	\[
		\varSigma
		= \mathop{\sum\sum\sum\sum\sum\sum\sum}_{\substack{
			1\leq r,t\leq p;\,
			1 \leq u, m_1, m_2, m_3, m_4 < p \\
			\mathclap{ p \nmid (m_1+u)(m_2+u)(m_3+u)(m_4+u) }
		}}
		\e[p]{ rA(\boldsymbol{m},0) + tA(\boldsymbol{m},u) },
	\]
	where $A(\boldsymbol{m},u)$ is given by~\cref{eq:Adef}.\\
	After moving the summation over $r$ and $t$ to the right and observing that
	\[
		\sum_{1\leq r\leq p} \e[p]{rx}
		= \begin{cases}
			p & \text{if } p\mid x, \\
			0 & \text{if } p\nmid x,
		\end{cases}
	\]
	we find that
	\begin{equation}\label{eq:CountingSolutions}
		\varSigma = p^2 \sum_{1\leq u<p} \#\mathscr{X}_{u},
	\end{equation}
	where $\mathscr{X}_{u}$ is given by~\cref{eq:Xu:def}. By \cref{lem:Diophantine}, $\varSigma\ll p^{5}$, so that from~\cref{eq:ContributionToAverage} we infer~\cref{eq:KloostermanTypePowerEstimate}.
\end{proof}
\begin{rem}\label{rem:Kloosterman}
	It should be stressed that in the proof of \cref{prop:KloostermanTypePowerEstimate} we do not obtain a non-trivial bound for an individual sum $S(x,y,w;p)$; it is only the averaging over $w\in\mathscr{W}$ which makes this result non-trivial. This may be compared with \citet{kloosterman1927ontherepresentation}, who obtains
	\begin{equation}\label{eq:KloostermansOriginalBound}
		\abs{ \Kl{p}{x}{y} } \ll p^{3/4} (x,y,p)^{1/4},
	\end{equation}
	(recall~\cref{eq:KloostermanSum}) by exploiting the transformation property
	\[
		\Kl{p}{x}{y} = \Kl{p}{ax}{\overline{a}y}
		\quad (\text{assuming } p\nmid a)
	\]
	in place of~\cref{eq:TransformationProperty}.
	Notice that only two parameters are transformed, so that when considering the obvious analogue $\varSigma_0$ of $\varSigma$, one only has to average over two parameters:
	\[
		\varSigma_0 = \mathop{\sum\sum}_{1\leq r,t\leq p} \abs{\Kl{p}{r}{t})}^4.
	\]
	Similarly as above one finds that
	\begin{equation}\label{eq:KloostermanOriginal}
		(p-1) \abs{ \Kl{p}{x}{y} }^4
		\leq \varSigma_0
		= p^2\, \#\braces{\text{solutions $\boldsymbol{m}\in(\ZZ/p\ZZ)^4$ to two congruences}}
	\end{equation}
	(see~\cite[pp.~423--236]{kloosterman1927ontherepresentation}, or \cite{heath-brown2000arithmeticappl} for a shorter exposition).
	Heuristically, one might expect that each congruence condition restricts one of the four variables, so the number of solutions counted on the right hand of~\cref{eq:KloostermanOriginal} ought to be $\ll p^2$. Indeed, this can be justified using arguments similar to those used in the proof of \cref{lem:Diophantine} (again, see \cite{kloosterman1927ontherepresentation,heath-brown2000arithmeticappl}) and from \cref{eq:KloostermanOriginal} one then concludes~\cref{eq:KloostermansOriginalBound}.
	We note that in contrast to the above heuristic, we are unable to show $\#\mathscr{X}_u \ll p^2$ for $\mathscr{X}_u$ given in~\cref{eq:Xu:def}, and numerical experiments seem to suggest that this is actually false; \cref{lem:Diophantine} only verifies our heuristic on average.
\end{rem}

\subsection{Proof of \texorpdfstring{\cref{thm:KloostermanBeattyBound}}{Theorem\autoref{thm:KloostermanBeattyBound}}}

We shall need the following result which is an easy consequence of the pigeonhole principle:
\begin{lem}[{\cite[Lemma 3.3]{banks2006nonresidues}}]
	\label{lem:BanksShparlinski:BigSet}Let $\alpha$ be a fixed irrational number. Then, for every positive integer $K$ and real number $\Delta\in\lparen0,1\rbrack$, there exists a real number $\gamma$ such that
	\[
		\#\braces{ k\leq K : \braces{k\alpha +\gamma} < \Delta } \geq 0.5\, K \Delta.
	\]
\end{lem}

\begin{proof}[{\normalfont To give a} proof of \cref{thm:KloostermanBeattyBound},]
	we adapt the reasoning from~\cite{banks2006nonresidues}.
	Let $K\leq N$ be a positive integer and $\Delta\in\lparen0,1\rbrack$ to be determined later (see~\cref{eq:Delta:and:K:choice} below). Then, by \cref{lem:BanksShparlinski:BigSet}, there is some real number $\gamma$ such that the set
	\[
		\mathscr{K} = \braces{ k\leq K : \braces{k\alpha+\gamma} < \Delta }
	\]
	has cardinality
	\begin{equation}\label{eq:K:size}
		\#\mathscr{K} \geq 0.5\, K \Delta.
	\end{equation}
	Furthermore, let
	\[
		\mathscr{N} = \braces{ 1\leq n\leq N : \braces{n\alpha+\beta-\gamma} < 1-\Delta }
	\]
	and $\mathscr{N}^{\mathrm{c}} = \braces{1,\ldots,N} \setminus \mathscr{N}$. Clearly, recalling~\cref{eq:Discrepancy},
	\begin{equation}
		\label{eq:Ncompl:size}
		\#\mathscr{N}^{\mathrm{c}} = N\Delta + O(ND_{\alpha,\beta}(N)).
	\end{equation}
	Now, writing
	\[
		x_{n,k} = \e*[p]{x\floor{(n+k)\alpha+\beta} + y\overline{\floor{(n+k)\alpha+\beta}}}
	\]
	for the moment, for every $k\in\mathscr{K}$, we find that
	\begin{align*}
		\K{p}{\alpha}{\beta}{N}{x}{y} &
		=\sum_{\substack{1\leq n\leq N\\
				p\nmid\floor{(n+k)\alpha+\beta}
		}} x_{n,k} + O(K) \\ &
		=\sum_{\substack{n\in\mathscr{N}\\
				p\nmid\floor{(n+k)\alpha+\beta}
		}} x_{n,k} + O\parentheses{ K + \#\mathscr{N}^{\mathrm{c}} }.
	\end{align*}
	Consequently,
	\begin{equation}
		\label{eq:K:averaging}
		\K{p}{\alpha}{\beta}{N}{x}{y}
		= \frac{W}{\#\mathscr{K}} + O(K+\#\mathscr{N}^{\mathrm{c}}),
	\end{equation}
	where
	\[
		W = \mathop{ \sum_{n\in\mathscr{N}} \sum_{k\in\mathscr{K}} }_{\mathclap{ p\nmid\floor{(n+k)\alpha+\beta} }} \e*[p]{ x\floor{(n+k)\alpha+\beta} + y\overline{\floor{(n+k)\alpha+\beta}} }.
	\]
	For any $(n,k)\in\mathscr{N}\times\mathscr{K}$, a simple calculation shows that
	\[
		\floor{ (n+k)\alpha+\beta }
		= \floor{ n\alpha+\beta-\gamma} + \floor{ k\alpha+\gamma }.
	\]
	Next, we apply Cauchy's inequality to $W$, getting
	\begin{align*}
		\abs{W}^{2} &
		\leq \#\mathscr{N} \cdot \sum_{n\in\mathscr{N}} \abs[\bigg]{ \! \sum_{\substack{
			k\in\mathscr{K} \\
			p\nmid\floor{(n+k)\alpha+\beta}
		}} \! \e*[p]{ x\floor{(n+k)\alpha+\beta} + y\overline{\floor{(n+k)\alpha+\beta}} } }^2 \\ &
		\ll_\alpha N \cdot \sum_{1\leq s\leq p} \, \abs[\bigg]{ \! \sum_{\substack{
			k\in\mathscr{K} \\
			p\nmid(s+\floor{ k\alpha+\gamma})
		}} \! \e*[p]{ x(s+\floor{ k\alpha+\gamma}) + y\overline{(s+\floor{ k\alpha+\gamma})} }}^2,
	\end{align*}
	where replacing the summation over $n\in\mathscr{N}$ by the summation over $s\leq p$ is allowed, since our assumption $N\leq p$ ensures that
	\[
		\#\braces{ n\in\mathscr{N} : \floor{ n\alpha+\beta-\gamma}\equiv s\bmod p }
		< 1 + \alpha
		\ll_\alpha 1.
	\]
	On squaring out the inner sum, we find that
	\[
		\abs{W}^{2}
		\ll_\alpha
		N \cdot \mathop{\sum\sum\sum}_{\substack{
				1\leq s\leq p, \: k,\ell\in\mathscr{K} \\
				\mathclap{ p\nmid(s+\floor{ k\alpha+\gamma}) } \\
				\mathclap{ p\nmid(s+\floor{\ell\alpha+\gamma}) }
		}}
		\begin{multlined}[t]
			\e[p]{x(\floor{ k\alpha+\gamma} + \floor{\ell\alpha+\gamma})} \times{} \\
			\shoveleft[.5cm]{ \times \e*[p]{y(\overline{s+\floor{ k\alpha+\gamma}} - \overline{s+\floor{\ell\alpha+\gamma}})}. \hfill }
		\end{multlined}
	\]
	This is
	\[
		\ll_\alpha N \cdot \mathop{\sum\sum}_{k,\ell\in\mathscr{K}} \, \abs[\bigg]{ \sum_{\substack{
			1\leq s\leq p \\
			p\nmid(s+\floor{ k\alpha+\gamma}) \\
			p\nmid(s+\floor{\ell\alpha+\gamma})
		}} \e*[p]{y(\overline{s + \floor{ k\alpha+\gamma}} - \overline{s + \floor{\ell\alpha+\gamma}})} }.
	\]
	Upon writing
	\[
		\mathscr{W}_k
		= \braces{ \floor{\ell\alpha+\gamma} - \floor{k\alpha+\gamma} : \ell \in \mathscr{K} },
	\]
	we infer
	\[
		\abs{W}^{2}
		\ll_\alpha N \adjustlimits\sum_{k\in\mathscr{K}}\sum_{w\in\mathscr{W}_k} \abs[\bigg]{ \sum_{\substack{ s<p \\ p\nmid(s+w) }} \e[p]{y(\overline{s}-\overline{s+w})} }.
	\]
	Hence, by \cref{prop:KloostermanTypeEstimate} and using $\#\mathscr{K} = \#\mathscr{W}_k$ ($k\in\mathscr{K}$),
	\begin{equation}\label{eq:WsquaredBound}
		\abs{W}^{2}
		\ll_\alpha N (\#\mathscr{K})^{7/4} p.
	\end{equation}
	In view of~\cref{eq:K:averaging}, we find that
	\[
		\abs{\K{p}{\alpha}{\beta}{N}{x}{y}}
		\ll_\alpha \sqrt{\frac{Np}{(\#\mathscr{K})^{1/4}}} + K + \#\mathscr{N}^{\mathrm{c}}.
	\]
	Consequently, upon gathering~\cref{eq:K:size},~\cref{eq:Ncompl:size} and \cref{lem:DiscrepBetaIndep}, we obtain the bound
	\[
		\abs{\K{p}{\alpha}{\beta}{N}{x}{y}}
		\ll_\alpha (Np)^{1/2} (K\Delta)^{-1/8} + K + N\Delta + ND_{\alpha}(N).
	\]
	We now choose $\Delta$ and $K$. Specifically, we aim at obtaining a bound of the shape $\abs{\K{p}{\alpha}{\beta}{N}{x}{y}} = o_\alpha(N)$ in a range $N\leq p<N^\theta$ with $\theta>1$ as large as possible.
	It can be seen that
	\begin{equation}
		\label{eq:Delta:and:K:choice}
		\Delta = N^{-105/128} p^{21/32}
		\quad\text{and}\quad
		K = \ceil{ N\Delta }
	\end{equation}
	yield the assertion of the theorem provided that $\Delta<1$, i.e., when
	\begin{equation}\label{eq:Hypothesis}
		(\log p)/\log N < \tfrac{5}{4}.
	\end{equation}
	On the other hand, if \cref{eq:Hypothesis} fails to hold, then the theorem asserts nothing more than the trivial bound
	\(
		\abs{\K{p}{\alpha}{\beta}{N}{x}{y}} \ll N
	\),
	so the proof is complete.
\end{proof}

\section{Improvements and open problems}
\label{sec:Improvements}

\subsection{Improvements}
As noted earlier in~\cref{rem:Improvements}, I.~E.~Shparlinski pointed out to the author that a generalised form of Weil's bound also applies to the sum $S(x,y,w;p)$ from~\cref{eq:Sum:Sxywp}:
indeed, by applying~\cite[Theorem~2]{moreno1991exponential-sums} to the rational function
\[
	R(m)
	= \frac{x}{m} + \frac{y}{m+w}
	= \frac{(x+y)m+xw}{m(m+w)}
	\in (\ZZ/p\ZZ)[m],
\]
it is straight-forward to see that
\begin{equation}\label{eq:ImprovedKloostermanTypeEstimate}
	S(x,y,w;p) \leq 2\sqrt{p}+1
\end{equation}
whenever $p\nmid x+y$ or $p\nmid xw$, for then $R(m)$ is (after cancelling any potential common factors in the above) a non-zero quotient with denominator of degree at most two and numerator of smaller degree.
The bound~\cref{eq:ImprovedKloostermanTypeEstimate} should be considered as an improved version of \cref{prop:KloostermanTypeEstimate}.

Consequently, we have a chain of improvements: in the proof of \cref{thm:KloostermanBeattyBound}, where we have used \cref{prop:KloostermanTypeEstimate}, we instead get
\begin{equation}
	\tag{$\text{\ref{eq:WsquaredBound}}'$}
	\abs{W}^{2} \ll_\alpha N (\#\mathscr{K}) \parentheses{p + (\#\mathscr{K}) \sqrt{p}}
\end{equation}
in place of~\cref{eq:WsquaredBound}.
Then, the choices in~\cref{eq:Delta:and:K:choice} can then be adapted to
\begin{equation}
	\tag{$\text{\ref{eq:Delta:and:K:choice}}'$}
	\Delta = N^{-1/4} p^{1/8}
	\quad\text{and}\quad
	K = \ceil{ N^{1/4} p^{3/8} },
\end{equation}
provided that we assume that $N\leq p < N^2$.
(Recall that the method requires $N\leq p$ and $\Delta<1$.)
This yields
\begin{equation}
	\tag{$\text{\ref{eq:KloostermanBeattyBound}}'$}
	\abs{\K{p}{\alpha}{\beta}{N}{x}{y}}
	\ll_\alpha N^{3/4} p^{1/8} + ND_{\alpha}(N)
\end{equation}
in place of~\cref{eq:KloostermanBeattyBound}.

Similar to our deduction of \cref{cor:KloostermanBeattyBound:FiniteType}, we may note that \cref{lem:FiniteTypeDiscrepancyBound} then implies the bound
\begin{equation}
	\tag{$\text{\ref{eq:KloostermanBeattyBound:FiniteType}}'$}
	\abs{\K{p}{\alpha}{\beta}{N}{x}{y}}
	\ll_\alpha N^{3/4} p^{1/8}
\end{equation}
in place of~\cref{eq:KloostermanBeattyBound:FiniteType}, provided that $\alpha$ is of finite type~$<8$.

Consequently, we have the following improved version of \cref{eq:HyperbolaSolutionDetector:Simplified}:
\begin{equation}
	\tag{$\text{\ref{eq:HyperbolaSolutionDetector:Simplified}}'$}
	(\log p) N^{-2+3/4} p^{1+1/8} \ll_{\alpha,\epsilon} 1.
\end{equation}
Together with the choice $N = \floor{ p^{9/10}\log p }$ this yields the following improved bound in \cref{thm:max:mm:bound}:
\begin{equation}
	\tag{$\text{\ref{eq:smallestSol}}'$}
	\max\braces{m,\tilde{m}} \leq \alpha p^{9/10}\log p + \beta,
\end{equation}
and the assumption on the type of $\alpha$ can be relaxed to allow for all finite types strictly smaller than~$8$.

\subsection{Scope and open problems}
Apart from the obvious advantage of yielding superior results, the above approach via Weil's bound has the disadvantage of failing to generalise to composite moduli in an obvious manner.
On the other hand, when studying \cref{question:LeastModHyperbolaPoint:Beatty} with $p$ replaced by some composite number, we expect that our approach via~\cref{prop:KloostermanTypeEstimate} should still be capable of yielding non-trivial results.
A suitable generalisation of \cref{lem:Diophantine} with composite~$p$ should be the main ingredient here.

Another direction in which our results may be generalised goes as follows: note that, as remarked earlier in \cref{sec:KloostermanBeattyBounds}, by \cref{lem:Weyl}, \cref{thm:KloostermanBeattyBound} is non-trivial for every irrational $\alpha>1$ provided only that $N$ and $p$ are large and in a suitable range.
In spite of this, we obtain a non-trivial answer to \cref{question:LeastModHyperbolaPoint:Beatty} (in the form of \cref{thm:max:mm:bound}) only for $\alpha$ with some restriction on its type.
The reason for this defect can be traced back to the $\log p$ factor appearing in~\cref{eq:HyperbolaSolutionDetector}, which in turn arose from the application of \cref{lem:CompletingTechnique}.
It seems plausible that one should be able to remove this by considering a suitable analogue of \cref{lem:CompletingTechnique} where the sharp cut-off at $a$ and $b$ in the sum
\[
	\sum_{a<m\leq b} A_m
\]
is smoothed out. Specifically, working instead with
\[
	\sum_{1\leq m\leq q} A_m \chi(m),
\]
where $\chi = \chi_0 * \chi_0$ is given as the convolution of the characteristic function $\chi_0$ of an interval with itself, seems promising.

\section*{Acknowledgements}
Most of this work is part of the author's doctoral dissertation at Würzburg University.
The author gratefully acknowledges the encouragement of his advisor, Jörn~Steuding.
The work on this article was completed while the author was employed at Graz University of Technology.
The financial support by both Würzburg University and Graz University of Technology is highly appreciated.
The author would like to thank Christoph~Aistleitner for comments on an earlier draft.
Furthermore, the valuable criticism of the anonymous referee has lead to various improvements on the exposition and is also greatly appreciated.


\end{document}